\documentclass[10pt]{article}
\usepackage{amsmath}
\usepackage{amssymb}
\usepackage[colorlinks=true]{hyperref}
\usepackage{amsthm}
\usepackage{cleveref}
\hypersetup{pdftex, colorlinks=true, linkcolor=maroon, citecolor=maroon,
  filecolor=blue,urlcolor=blue}

\usepackage{graphicx}
\usepackage{url}
\usepackage{color}
\definecolor{maroon}{rgb}{.69,.188,.376}
\definecolor{darkgreen}{rgb}{0,.5,0}
\definecolor{darkblue}{rgb}{0,0,.5}
\definecolor{magenta}{rgb}{1,0,1}

\usepackage[mathscr]{euscript}		
\usepackage{dsfont}
\usepackage{psfrag}			

\usepackage{caption} 

\usepackage{anysize}
\usepackage{cleveref}

\usepackage{enumerate}
\usepackage{enumitem}

\usepackage[top=.95in, bottom = 1 in, left=1in, right = .6in]{geometry}

\newcommand{\Z}{\ensuremath{\mathbb{Z}}}
\newcommand{\N}{\ensuremath{\mathbb{N}}}
\newcommand{\R}{\ensuremath{\mathbb{R}}}

\newcommand{\E}{\ensuremath{\mathbb{E}}}
\renewcommand{\P}{\ensuremath{\mathbb{P}}}

\newtheorem{theorem}{Theorem}[section]
\newtheorem{lemma}[theorem]{Lemma}

\newtheorem{proposition}[theorem]{Proposition}
\newtheorem{remark}[]{Remark}

\numberwithin{equation}{section}

\definecolor{Red}{rgb}{1,0,0}
\definecolor{Blue}{rgb}{0,0,1}
\definecolor{Olive}{rgb}{0.41,0.55,0.13}
\definecolor{Yarok}{rgb}{0,0.5,0}
\definecolor{Green}{rgb}{0,1,0}
\definecolor{MGreen}{rgb}{0,0.8,0}
\definecolor{DGreen}{rgb}{0,0.55,0}
\definecolor{Yellow}{rgb}{1,1,0}
\definecolor{Cyan}{rgb}{0,1,1}
\definecolor{Magenta}{rgb}{1,0,1}
\definecolor{Orange}{rgb}{1,.5,0}
\definecolor{Violet}{rgb}{.5,0,.5}
\definecolor{Purple}{rgb}{.75,0,.25}
\definecolor{Brown}{rgb}{.75,.5,.25}
\definecolor{Grey}{rgb}{.7,.7,.7}
\definecolor{Black}{rgb}{0,0,0}

\newcommand{\ignore}[1]{{}}

\date{\today}
\begin{document}
\baselineskip=14pt
\setkeys{Gin}{width=\textwidth}

\title{Elliptic Harnack Inequality for  $\Z^d$}

\author{
Siva Athreya%
  \thanks{ Indian Statistical Institute, 8th Mile Mysore Road, Bangalore 560059, India.
    Email: \url{athreya@isibang.ac.in}}
  \and
Nitya Gadhiwala
  \thanks{Department of Mathematics, University of British Columbia, 1984 Mathematics Road, Vancouver, BC Canada, V6T 1Z2     Email: \url{nitya.g20@gmail.com}}
  \and
 Ritvik R. Radhakrishnan%
  \thanks{Department of Mathematics, ETH Zurich, HG G 33.4, Raemistrasse 101 8092 Zurich Switzerland
    Email: \url{ritvik.isibang@gmail.com}}
}

\maketitle

\begin{abstract}
  We prove the scale invariant Elliptic Harnack Inequality (EHI) for
  non-negative harmonic functions on $\Z^d$. The purpose of this
  note is to provide a simplified self-contained
  probabilistic proof of EHI in $\Z^d$  that is accessible at the undergraduate
  level.  We use the Local Central Limit Theorem for simple symmetric
  random walks on $\Z^d$ to establish Gaussian bounds for the $n$-step
  probability function. The uniform Green inequality and the classical
  Balayage formula then imply the EHI.
\end{abstract}

\noindent {\em AMS 2010 Subject Classification :} Primary: 05C81 Secondary: 31C05, 31C20 \\
\noindent {\em Keywords :} Random walk, Harmonic Function, Harnack Inequality, Gaussian Bounds, Balayage.

\section{Introduction}

The scale invariant Elliptic Harnack Inequality (EHI) and its
applications to the theory of elliptic and parabolic partial
differential equations are well known. The work on regularity of
solutions dates back to 1950's in the works of de Giorgi \cite{Gi57},
Nash \cite{N58} and Moser \cite{Mo61}. Moser first proved EHI for
solutions to partial differential equations in divergence form in
$\R^d$. Bombieri-Giusti \cite{BG72} showed that EHI holds for
minimal surfaces and EHI was proven for metric spaces by Sturm 
\cite{S98}. We refer the interested reader to the survey article by
Kassmann \cite{K07}, and the references therein for a comprehensive
review of the  Harnack inequality (for elliptic and parabolic
operators) in $\R^d$.

The EHI was established for weighted graphs by Delmotte
\cite{D99}. Barlow \cite{MB17} has a comprehensive treatment of random
walks on locally finite graphs and in \cite[Theorem 7.19]{MB17}
presents a proof of EHI for weighted graphs that have controlled
weights and are roughly isometric to $\Z^d$.  In this short note our
aim is to provide a simplified self-contained probabilistic proof of
EHI for harmonic functions on $\Z^d$ using properties of the simple symmetric random walk. Lawler in \cite[Theorem 1.7.2]{GL13} also provides a proof of EHI for $\Z^d$. In this short note our aim is to provide a simplified self-contained probabilistic proof of
EHI for harmonic functions on $\Z^d$ following the framework laid out in \cite{MB17} using properties of the simple symmetric random walk.

Let $d \geq 1$ and $\Z^d$ be the integer lattice. We will think of $\Z^d$ as a graph with vertex set $$\Z^d=\{(x_i)_{i=1}^d: x_i \in \Z\}$$ endowed with the (graph) distance  $d(x,y) = \sum_{i=1}^d |x_i - y_i|$ between any two points $x,y \in \Z^d$. We may view edge set $E = \{ \{x,y\}: x,y \in V, \, d(x,y) = 1\}.$  For any $x,y \in \Z^d$ we shall say $x \sim y$ if $\{x,y\} \in E$.   For $r >0$, we define the ball of radius $r$ around a point $x_0 \in \Z^d$ as
$$ B(x_0,r) = \{ x \in  \Z^d : d(x,x_0) \leq r\}.$$ Let $A
\subset \Z^d$.  We define the boundary of $A \subset \Z^d$ by $$\partial A = \{ y \in A^c : \exists \, x \in A \mbox{ with } x \sim y \}$$
and set $\bar{A}  = A \cup \partial A$.  For $f: \bar{A} \rightarrow \R$ we define the Laplacian $\Delta$  of $f$ as
    $$\Delta f(x) = \frac{1}{2d} \sum_{x \sim y} (f(x)-f(y)), \qquad \mbox{for all }  x \in A.$$  A function $f : \bar{A} \rightarrow \R$ is said to be harmonic in
$A$ if $\Delta f (x) = 0$ for all $x \in A$.  The scale invariant EHI is then stated as follows.

\begin{theorem}[{\bf EHI}] \label{t:ehi}Let $ d\geq 1$.  There exists $C$ such that for any $x_0 \in \Z^d$, $ R\geq 1$, $h: \overline{B(x_0, R)} \rightarrow \R_+$ with $h$ harmonic in $B(x_0, R)$ we have
  \begin{equation}\label{eq:ehi}
    \max_{y \in B\left(x_0,\frac{R}{2}\right)} h(y) \leq C \min_{z\in
      B\left(x_0,\frac{R}{2}\right)}h(z)
    \end{equation}
\end{theorem}

 We work with simple symmetric random walks in $\Z^d$, whose
 transition probability matrix $P$ is connected with the Laplacian
 $\Delta$. The proof has three key steps -- the Gaussian bounds for
 the $n$-step probability function of the simple symmetric random walk
 (Proposition \ref{p:gb}), the uniform Green inequality for the Green
 function asssociated with the random walk (Proposition \ref{p:ugi}),
 and the Balayage formula for harmonic functions (Proposition
 \ref{p:blyge}).

 The proof of Theorem \ref{t:ehi} that we present here proceeds along
 the framework given in the proof of \cite[Theorem 7.19]{MB17} (which
 applies for those graphs which have controlled weights and are
 roughly isometric to $\Z^d$). Our proof uses methods specific
     to the random walk on $\Z^d$ and has appropriate simplifications
     whenever possible. To establish Proposition \ref{p:gb} for
     weighted graphs one first shows the Poincar\'e inequality along with
     a growth condition for volume of the balls in the graph. Then the
     equivalence results shown in \cite{D99} will imply the result
     (See also \cite[Theorem 6.28]{MB17}). For $\Z^d$, we use the Local Central
     Limit Theorem for simple random walks to obtain near diagonal
     bounds on the $n$-step transition probability function (See
     Proposition \ref{p:lclt}). Then the Gaussian lower bound is
     obtained via a classical chaining of balls argument. The Gaussian
     upper bound follows from exponential bounds for the distribution
     of exit times of the simple random walk from a ball. For this, we
     use the specific structure of the walk in $\Z^d$, by reducing the
     question for the $d$ dimensional random walk to a one dimensional
     walk which follow from Chernoff bounds (See Lemma
     \ref{l:hbound}). We also note that Gaussian bounds for the simple
     symmetric random walk in $\Z^d$ are part of folklore and are well
     known but it is hard to point to explicit references in the
     literature where they can be found.
    
As mentioned before,  \cite[Theorem 1.7.2]{GL13} also contains a proof for the EHI in $\Z^d$.  There the uniform Green inequality is obtained using explicit bounds on the hitting distribution of the walk for three and higher dimensions \cite[Proposition 1.5.10]{GL13} and on the Green function for two dimensions in \cite[Proposition 1.6.7]{GL13}. The Local Central Limit Theorem is used to derive precise asymptotics on the Green function and the  propositions are proved via an application of the optional stopping theorem for martingales. For proving the uniform Green inequality we use Proposition
     \ref{p:gb} to establish Gaussian bounds for $n$-step transition
     probabilities of walk killed on exiting the ball. Plugging these
     estimates and exit time bounds obtained above into the definition
     of the Green function one obtains the result. The Balayage
     formula is shown using the probabilistic representation for the
     solution to the Dirichlet problem and standard potential theory
     techniques. The proofs of uniform Green inequality and the Balayage formula  translate to graphs that have controlled weights and are roughly isometric to $\Z^d$ (See \cite[Theorem 4.26, Theorem 7.15]{MB17}).

The Balayage formula is analogous to the classical Poisson integral
formula for Harmonic functions in $\R^d$ with the Green function
playing the role of the Poisson Kernel. The uniform Green inequality
then enables the scale invariant Harnack inequality. The uniform Green
inequality are implied by the Gaussian bounds.

We remark that the constant $C$ in Theorem \ref{t:ehi} does not depend on $R$ but only on dimension $d$ and the $\frac{1}{2}$-scaling used to obtain the smaller ball. There are many applications of the EHI. For instance, using EHI one can show that the probability of exiting the boundary of $B(x_0,R)$ via a particular set is comparable for different starting points inside $B(x_0,\frac{R}{2}$).  The EHI implies  the oscillation inequality for  harmonic functions, in the sense that if for any set $A \subset \Z^d,$ $$ \mbox{Osc}(h,A) = \sup_{x \in A}h(x)-\inf_{x \in A}h(x)$$ then there is a $0 <\delta <1$ such that 
$$ \mbox{Osc}\left(h,B\left(x_0,\frac{R}{2}\right)\right) \leq (1-\delta) \mbox{Osc}\left(h,B(x_0,R)\right).$$ 
The  oscillation inequality will then imply H\"older continuity for $h$ and more importantly the Strong Liouville property, which asserts that if $h \geq 0$ and is harmonic on $\Z^d$ then $h$ is constant (See \cite[Theorem 1.46]{MB17}). 

\medskip
{\bf Acknowledgements:} We would like to thank Martin Barlow and D. Yogeshwaran for useful discussions and feedback on a draft of the note.  We also thank Greg Lawler for pointing us to \cite[Theorem 1.7.2]{GL13}.  This work was done as part of a summer reading seminar when Nitya Gadhiwala  and Ritvik Radhakrishnan were undergraduate students in the B. Math. (Hons.) program at the Indian Statistical Institute, Bangalore centre.  Siva Athreya was partially supported by a CPDA grant at the Indian Statistical Institute, Bangalore centre.

\medskip

\noindent {\bf Layout for rest of the paper:} In Section \ref{sec:ehiproof} we will prove
Theorem \ref{t:ehi} assuming the three Propositions. Then we prove the
Gaussian bounds in Section \ref{sec:ghb}, the uniform Green inequality
in Section \ref{sec:ugi}, and the Balayage formula in Section
\ref{sec:blyge}, to complete the proof.

\medskip
 
\noindent \textbf{Convention on constants}. Throughout the article $c,
c_1, c_2, \ldots$ denote positive constants whose value may change
from line to line. They are all positive and their precise values are
not important. In statements of results will use capital letters $A,B,
\ldots$ to denote special constants. Also, throughout this note the
constants depend on the dimension $d$, but will not be mentioned
explicitly. Any other dependence will be stated.

\section{Proof of Theorem \ref{t:ehi}} \label{sec:ehiproof}
When $d=1$, it is easy to see that any harmonic function $h: \Z \rightarrow \R$ is of the form $h(k) = \alpha k + \beta$ for some real numbers $\alpha, \beta.$ Then \eqref{eq:ehi} follows immediately. We will now present a proof when $d \geq 2.$

Let $X$ be the simple random walk on $\Z^d$.  That is,  $X$ is a  discrete time Markov chain $(\{X_n\}_{n \geq 0}, \{\P^{x}\}_{x\in \Z^d})$ with  transition matrix $P=[p_{xy}]$ given by 
\[ p_{xy} = \left \{ \begin{array}{ll} \frac{1}{2d} & \mbox{ if } x \sim y, \\
  0 & \mbox{ otherwise,}
\end{array} \right.\]
for any $x,y \in \Z^d$ on a measure space $(\Omega, {\cal F})$. Under the probability $\P^{x}$ we have $\P^x(X_0=x)=1$.  Note that $\Delta= (P-I)$. For any $n \geq 1$, we will denote the $n$-step probability function  of $X$ by
\begin{equation}\label{eq:nst}
p_n(x,y) = \P^x(X_n = y),
\end{equation}
with $p_1(x,y) = p_{xy}$. The first step in the proof is to establish the following Gaussian upper and lower bounds for the $n$-step  probabilities of the random walk.

\begin{proposition}[{\bf Gaussian Bounds}] \label{p:gb} Let $X_n$ be the simple random walk on $\Z^d$ with $n$-step  probability function $p_n(\cdot, \cdot)$ given by \eqref{eq:nst}. Then
  \begin{enumerate}
\item[(a)] there exists $L_1 >0$ and $L_2 >0$ such that for all $x,y \in \Z^d, n \geq \max\{1,d(x,y)\}$
        \begin{equation}\label{eq:lhkb}
      p_n(x,y) +p_{n+1}(x,y) \geq \frac{L_1}{n^{\frac{d}{2}}}\exp\left(-L_2\frac{d^2(x,y)}{n}\right),     \end{equation}
  \item[(b)]there exists $U_1 >0$ and $U_2 >0$ such that for all $x,y \in \Z^d, n \geq 0$
        \begin{equation}\label{eq:uhkb}
          p_n(x,y)  \leq \frac{U_1}{(\max\{n,1\})^{\frac{d}{2}}}\exp\left(-U_2\frac{d^2(x,y)}{\max\{n,1\}}\right).     \end{equation}
        \end{enumerate}
\end{proposition}

The second step in the proof is to use the Gaussian bounds to
establish the Uniform Green Inequality. Now, for any subset $D \subset
\Z^d$ define the entrance time to $D$ and exit time from $D$
\begin{equation*}
  T_D = \inf\{ k \geq 0: X_k \in D\} \text{ and }\tau_D = \inf\{ k \geq 0: X_k \in D^c\}.
  \end{equation*}
We define the  killed $n$-step  probability function of $X$ by
\begin{equation} \label{knst}
p^D_n(x,y) = \P^x(X_n = y, n < \tau_D).
  \end{equation}
The Green function corresponding to $D$, $g_D: D \times D \rightarrow \R$, is given by
\begin{equation}\label{eq:greenf}
g_D(x,y)= \sum_{n=0}^\infty p_n^D(x,y). 
\end{equation}
Thus the Green function represents the expected time spent (or number of visits) by the walk at vertex $y$ before exiting the domain $D$ when it starts at $x$. We now state the Uniform Green Inequality.
\begin{proposition}[{\bf Uniform Green Inequality}] \label{p:ugi} There exists $G_1,$  $G_2>0$ such that if $R \geq 2, B=B(x_0,R), x,y \in B(x_0, \frac{R}{2})$, $ r = \max\{1, d(x,y)\}$, then
  \begin{enumerate}
    \item[(a)] for $d \geq 3$
  \begin{equation}\label{eq:ugid}
    \frac{G_1}{r^{d-2}} \leq g_B(x,y) \leq \frac{G_2}{r^{d-2}},
    \end{equation}

  \item[(b)] and for $d=2$
    \begin{equation}\label{eq:ugi2}
  G_1\log\left(\frac{R}{r}\right) \leq g_B(x,y) \leq G_2 \log\left(\frac{R}{r}\right).
    \end{equation}
\end{enumerate}
\end{proposition}

The uniform Green inequality is a key tool in the proof of EHI due to
the fact that Green function plays a role similar to that of the
Poisson Kernel for harmonic functions in $\R^d$.  This established by
the Balayage formula for harmonic functions. To state the result, we
need one additional notation. For any $A \subset \Z^d$, we define the
interior boundary of $A$ by
 $$\partial_i A = \{ y \in A : \exists \, x \in A^c \mbox{ with } x \sim y \}.$$

\begin{figure}
\includegraphics[scale=0.1]{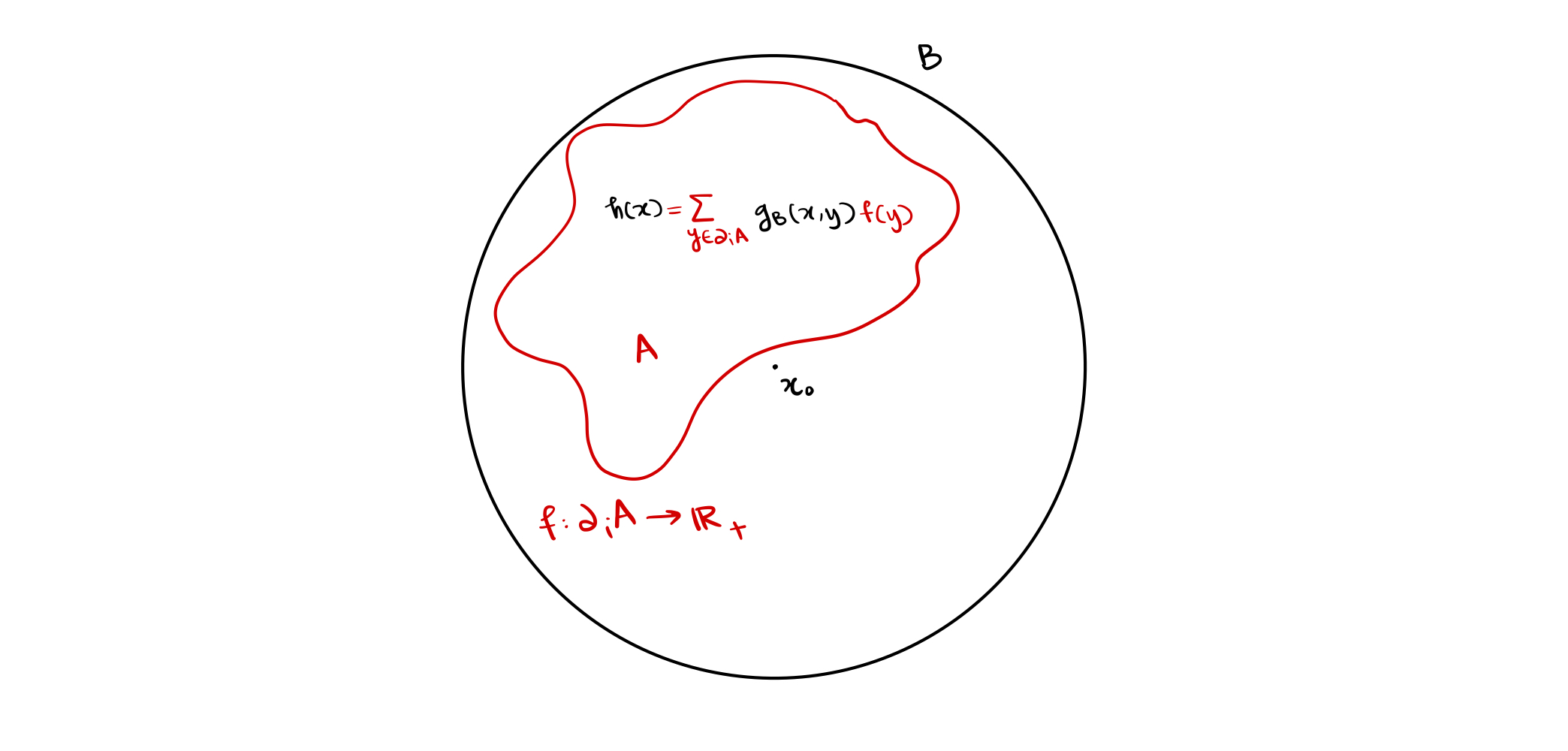}
  \caption{(Balayage Formula) For  harmonic function $h: B \rightarrow R$ and a $A\subset B$ one can find a $f:\partial_i A \rightarrow \R$ such that $h$ is a weighted sum of $f$ in $A$ with the weights given by the Green function.}
\end{figure}

 \begin{proposition}[{\bf Balayage formula}] \label{p:blyge}
  Let $R \geq 1$, $x_0 \in V$, $B= B(x_0,R), A \subsetneq  B$.  Assume $h: \overline{B} \rightarrow \R_+$ with $h$ being harmonic in $B.$ Then there exist $f : \partial_i A \rightarrow \R_+$ such that
  \begin{equation} \label{eq:blyge}
    h(x) = \sum_{y \in \partial_i A} g_B(x,y) f(y)\,\quad \forall \, x \in A
    \end{equation}
\end{proposition}

Assuming the above Propositions we are now ready to prove the main
result. The proof is essentially an application of the Balayage
formula and the uniform Green inequality, followed by a standard
chaining of balls argument in $\Z^d$.

\begin{proof}[Proof of Theorem \ref{t:ehi}] We will first consider the case when $R \leq 32$. For  $x, y \in B(x_0, \frac{R}{2})$ such that $x \sim y$ we have
  \begin{equation} \label{eq:slehi}
    h(x) = \sum_{z \sim x} \frac{1}{2d}h(z) \geq \frac{1}{2d} h(y).
    \end{equation}
  Now for arbitary $x_1, x_2 \in B(x_0, \frac{R}{2})$ there exists $z_1, z_2, \ldots z_k$  such that $x_1 \sim z_1 \sim z_2 \ldots z_k \sim y$ with $k \leq 32.$ Using \eqref{eq:slehi}, we have that
$$  h(x_1)  \geq \frac{1}{(2d)^k} h(x_2) \geq \frac{1}{(2d)^{32}} h(x_2) .$$
  As $x_1, x_2$ were arbitrary,
 \begin{equation} \label{eq:mehi}
   \max_{y \in B\left(x_0,\frac{R}{2}\right)} h(y) \leq (2d)^{32} \min_{z\in B\left(x_0,\frac{R}{2}\right)}h(z).
 \end{equation}
Thus  proving \eqref{eq:ehi} whenever $R \leq 32$.  We shall now assume $R > 32$. Set $B=B(x_0,R), A = B\left (x_0, \frac{R}{2} \right).$ Let $x_1,x_2 \in B\left(x_0, \frac{R}{4}\right)$ and $y \in \partial_i A$. Firstly, $$\frac{R }{4} \leq d(x_i,y) \leq \frac{3 R }{4} \qquad \mbox{for $ i = 1,2$} . $$

Applying Proposition \ref{p:ugi}, by \eqref{eq:ugi2} or  \eqref{eq:ugid}, we have that \begin{equation}\label{eq:gh} g_B(x_1,y) \leq \frac{3^{d-2}G_1}{G_2}  g_B(x_2,y) \qquad \mbox{ for all } y \in \partial_i(A).\end{equation}
Using \eqref{eq:blyge} from Proposition \ref{p:blyge}, we have that there exists $f$ supported on $\partial_i A$ such that
\begin{equation*}
  h(x) = \sum_{y \in \partial_i A} g_B(x,y) f(y) \,\quad \forall \, x \in A.
\end{equation*}
Applying \eqref{eq:gh} in the above we have that 
\begin{equation}
h(x_1)= \sum_{y \in \partial_i A} g_B(x_1,y) f(y) \leq  \frac{3^{d-2}G_1}{G_2} \sum_{y \in \partial_i A} g_B(x_2,y) f(y) =  \frac{3^{d-2}G_1}{G_2} h(x_2). \label{eq:hsm}
\end{equation}
whenever $x_1, x_2 \in B(x_0, \frac{R}{4}).$  To complete the proof we use a standard chaining of balls argument.  As $R >32$, it is easy to see that there exists $N \geq 1$ such that for  all $u,v \in B(x_0, \frac{R}{2})$  there exists $z_1, z_2, \ldots, z_N \in B(x_0, \frac{R}{2})$ such that
\begin{figure}
\includegraphics[scale=0.2]{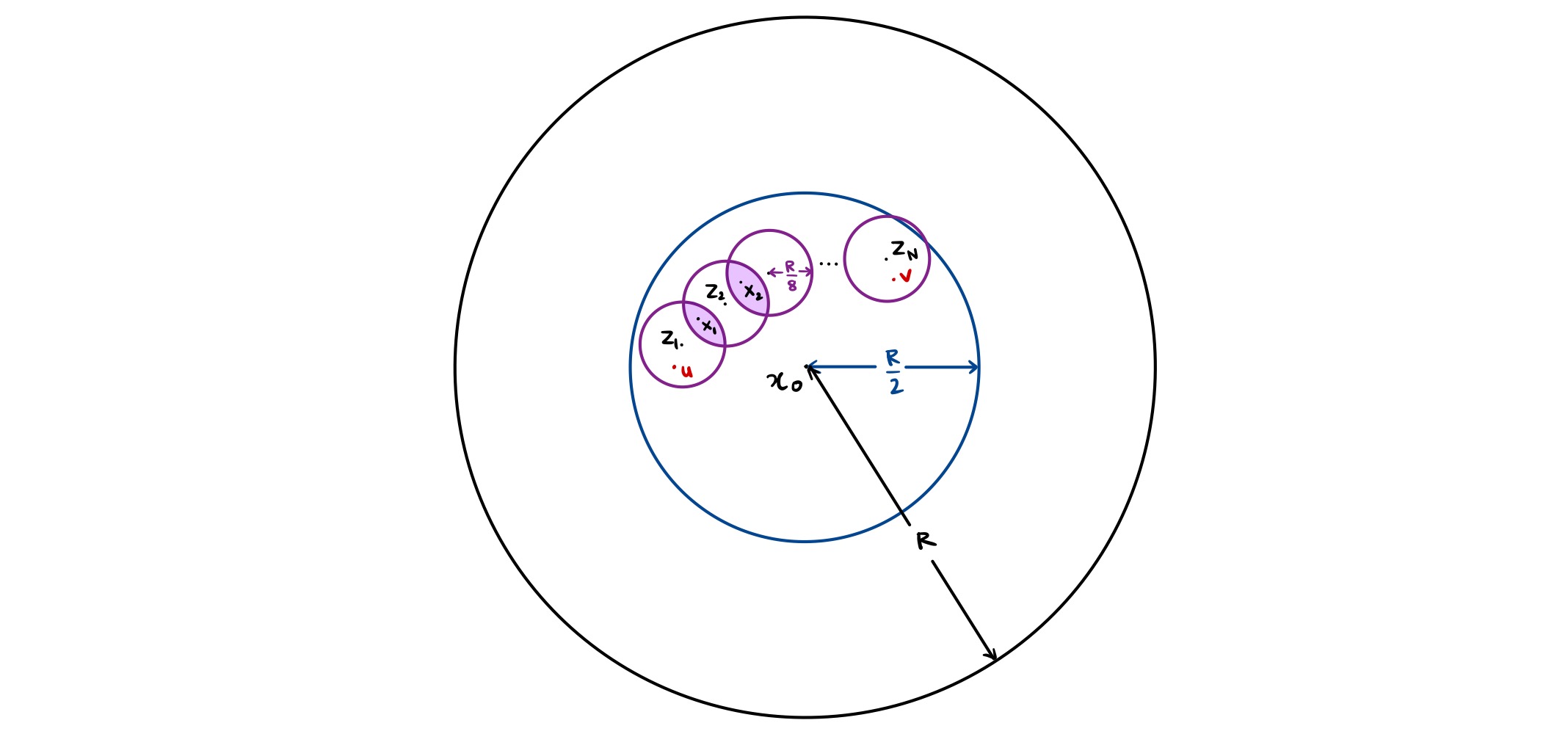}
  \caption{Figure depicts that standard chain of balls argument used in the proof of Theorem \ref{t:ehi}.}
  \end{figure}
  \begin{align*}    
 u \in B\left(z_{1}, \frac{ R }{8}\right),\qquad  v \in B\left(z_{N}, \frac{ R }{8}\right), \qquad  \mbox{ and }  \qquad B\left(z_{j}, \frac{ R }{8}\right) \cap B\left(z_{{j+1}}, \frac{ R }{8}\right) \neq  \emptyset \mbox{ for $1 \leq j \leq N-1.$}
  \end{align*}
Let us choose corresponding arbitrary points $x_{j} \in B(z_{j}, \frac{ R }{8}) \cap B(z_{{j+1}}, \frac{ R }{8})$. Note that $B(z_{j},\frac{ R }{2}) \subset B(x_0, R)$ for all $1 \leq j \leq N-1$. Thus $h$ satisfies the hypothesis when $R$ is replaced by $\frac{R}{2}$ and  $x_0$ is replaced by $z_{j}$. Applying \eqref{eq:hsm}   with $R$ replaced by $\frac{R}{2}$ and  $x_0$ successively replaced by $z_{j}$  we have that  $$ h(u) \leq  \frac{3^{d-2}G_1}{G_2}  h(x_1) \leq   \left ( \frac{3^{d-2}G_1}{G_2} \right)^2 h(x_2)\leq \left ( \frac{3^{d-2}G_1}{G_2} \right)^3 h(x_3) \leq \ldots  \left ( \frac{3^{d-2}G_1}{G_2} \right)^N h(v).$$
As $u,v$ were arbitary we have that, 
 \begin{equation} \label{eq:mehi1}
   \max_{y \in B\left(x_0,\frac{R}{2}\right)} h(y) \leq   \left ( \frac{3^{d-2}G_1}{G_2} \right)^N  \min_{z\in B\left(x_0,\frac{R}{2}\right)}h(z).
 \end{equation}
 We have thus established \eqref{eq:ehi} whenever $R >32$. This completes the proof.
 \end{proof}

\section{Gaussian Bounds} \label{sec:ghb}

In this section we prove Proposition \ref{p:gb} which provides Gaussian
upper and lower bounds for the $n$-step probability function of the simple
random walk. First  we shall use the well known Local Central
Limit Theorem to obtain bounds for the $n$-step probability function near the diagonal. Using this and a chaining of balls argument we will obtain the Gaussian lower bound. For the Gaussian upper bound, we use the near diagonal bounds obtained along with an application of Chernoff bound to a one-dimensional projection of the random walk. 

\subsection{Near diagonal bounds} \label{sec:lclt}

The following is the version of the local central limit theorem we use. We shall say that $n$ and $z \in \Z^d$ have the same parity if $n+|z_1|+\cdots+|z_d|$ is even.

\begin{theorem}[{\bf Local Central Limit Theorem}] \label{thm:lclt_lawler}
Let $\{X_n\}_{n\geq0}$ denote a simple symmetric random walk in $\mathbb{Z}^d$. There exists positive constants $A,B,C$ such that for all positive integers $n$ and points $x,y\in \Z^d$ we have
\begin{equation}
\left| p_n(x,y) -\frac{ A}{n^{\frac{d}{2}}} \exp\left(-\frac{B d(x,y)^2}{n}\right) \right| \leq \frac{C}{n^{\frac{d}{2}+1}},
\end{equation}
whenever $n$ and $y-x$ have the same parity.
\end{theorem}

\begin{proof}
We  refer the reader to Theorem 1.2.1 in \cite{GL13} for the proof (or for more detailed explanation of the same in \cite{g20}). We note that the statement in \cite{GL13} has $d(x,y)$ to be standard Euclidean metric namely  $ || \cdot ||_2$. The above statement where $d(x,y)$ is the graph distance easily follows because we have $c_1||x-y||_2\leq d(x,y) \leq c_2 ||x-y||_2$ for some constants $c_1>0$ and $c_2>0$. 
\end{proof}

\begin{proposition}[{\bf Near diagonal bounds}] \label{p:lclt}
    There exists $N_1 >0$ such that for all $x,y \in \Z^d, n \geq 0$
    \begin{equation} \label{eq:ulclt}
      p_n(x,y) \leq \frac{N_1}{(\max\{n,1\})^{\frac{d}{2}}},
    \end{equation}
    and
   there exists $N_2 >0$ and $L \in (0,1)$ such that for all $x,y \in \Z^d, n \geq \frac{1}{L^2} \max\{1,d^2(x,y)\}$
        \begin{equation}\label{eq:llclt}
      p_n(x,y) +p_{n+1}(x,y) \geq \frac{N_2}{n^{\frac{d}{2}}}. \end{equation}
  \end{proposition}

\begin{proof}
We first establish equation \eqref{eq:ulclt} assuming  Theorem \ref{thm:lclt_lawler}. Choose $x,y\in\mathbb{Z}^d$. We want to bound $p_n(x,y) = \mathbb{P}^x(X_n = y)$ from above. Using the Local Central Limit Theorem we get
\begin{align*}
p_n(x,y) &\leq \frac{C}{n^{\frac{d}{2}+1}} + \frac{ A}{n^{\frac{d}{2}}} \exp\left(-\frac{B d(x,y)^2}{n}\right)  \leq \frac{C+A}{n^{\frac{d}{2}}}.
\end{align*}
Thus the proof of \eqref{eq:ulclt} is complete.

Next we show \eqref{eq:llclt}. Choose $x = (x_1,\ldots,x_d),y = (y_1,\ldots,y_d)\in\mathbb{Z}^d$. We want find a lower bound for $p_n(x,y)+ p_{n+1}(x,y)= \mathbb{P}^x(X_n = y)+ \mathbb{P}^x(X_{n+1}= y)$. Since the simple random walk on $\mathbb{Z}^d$ has period $2$, exactly one of $p_n(x,y)$ and $p_{n+1}(x,y)$ will be nonzero. The first is nonzero  if $n+|y_1-x_1|+\cdots+|y_d-x_d|$ is even and the second if it is odd, in which case, $n+1+|y_1-x_1|+\cdots+|y_d-x_d|$ is even. Let $m$ be either $n$ or $n+1$ such that $p_m(x,y)$ is nonzero, so that $m$ and $y-x$ have the same parity. 

It suffices to show that there exists $c>0$ and $L\in(0,1)$ such that for all $x,y \in \Z^d$ and $n\geq \frac{1}{L^2}\max\{1,d^2(x,y)\}$, we have  $$p_{m}(x,y) \geq \frac{c}{n^{\frac{d}{2}}}.$$ Using the Local Central Limit Theorem  we have
\begin{align}
p_m(x,y) &\geq \frac{ A}{m^{\frac{d}{2}}} \exp\left(-\frac{B d(x,y)^2}{m}\right)  -\frac{C}{m^{\frac{d}{2}+1}} \nonumber \\
&\geq \frac{ 1}{n^{\frac{d}{2}}} \left(c_1 \exp\left(-\frac{c_2 d(x,y)^2}{n}\right) -\frac{c_3}{n} \right). \label{eq:upperbound}
\end{align} 

Now, let $n \geq \frac{1}{L^2} \max\{1,d^2(x,y)\}$, for some $L\in (0,1)$ that will be determined later. So we have $L^2 \geq \frac{1}{n} \max\{1,d^2(x,y)\}$. Replacing this in \eqref{eq:upperbound} we get
\begin{equation}
p_m(x,y)  \geq \frac{1}{n^{\frac{d}{2}}} \left( c_1\exp(-c_2L^2) -c_3L^2 \right) \nonumber
\end{equation}

Now choose $0 < \delta <1 $ such that for $0<L<\delta$ we have $$c_1\exp(-c_2L^2)>\frac{c_1}{2} \qquad \mbox{ and }\qquad  c_3L^2<\frac{c_1}{4}.$$ So for $0<L<\delta$ we get

\begin{equation*}
p_m(x,y)  \geq \frac{c_1}{2n^{\frac{d}{2}}}. 
\end{equation*}
The proof of \eqref{eq:llclt} is complete.
\end{proof}

\subsection{Lower Bound}
To establish the Gaussian lower bound we consider three cases, divided according to how $R=d(x,y)$ compares with $n$. For the third case we apply a chaining argument, and for this we require the following estimate on the volume of a ball in $\Z^d$. Namely, there exists a constant $V_1\leq 2d$ such that for all $x_0 \in \Z^d$ and $r\geq 1$ we have
\begin{equation} \label{eq:va}
 \left|B(x_0,r)\right|\geq V_1 r^d.
\end{equation}

The below proof follows the same structure as in the proof of \cite[Proposition 4.38]{MB17}.
\begin{proof}[Proof of Proposition~\ref{p:gb}(a)]

Let $x,y \in \Z^d, R=d(x,y)$ and let $L >0$ be as in \eqref{eq:llclt}.

\bigskip
 
{\bf Case 1: } Let  $R \leq n \leq \frac{2^6}{L^2}R$. Then,
$$p_{n}(x,y)+ p_{n+1}(x,y)  \geq \left (\frac{1}{2d} \right)^{n+1} \geq c_1 \exp(-c_2 n)  \geq \frac{c_1}{n^{\frac{d}{2}}}\exp\left(-c_3\frac{R^2}{n}\right).$$
This implies \eqref{eq:lhkb}. 

\bigskip
{\bf Case 2:} Let $n \geq \frac{R^2}{L^2}$. In this case \eqref{eq:llclt} implies \eqref{eq:lhkb}.

\bigskip

\begin{figure} 
  \includegraphics{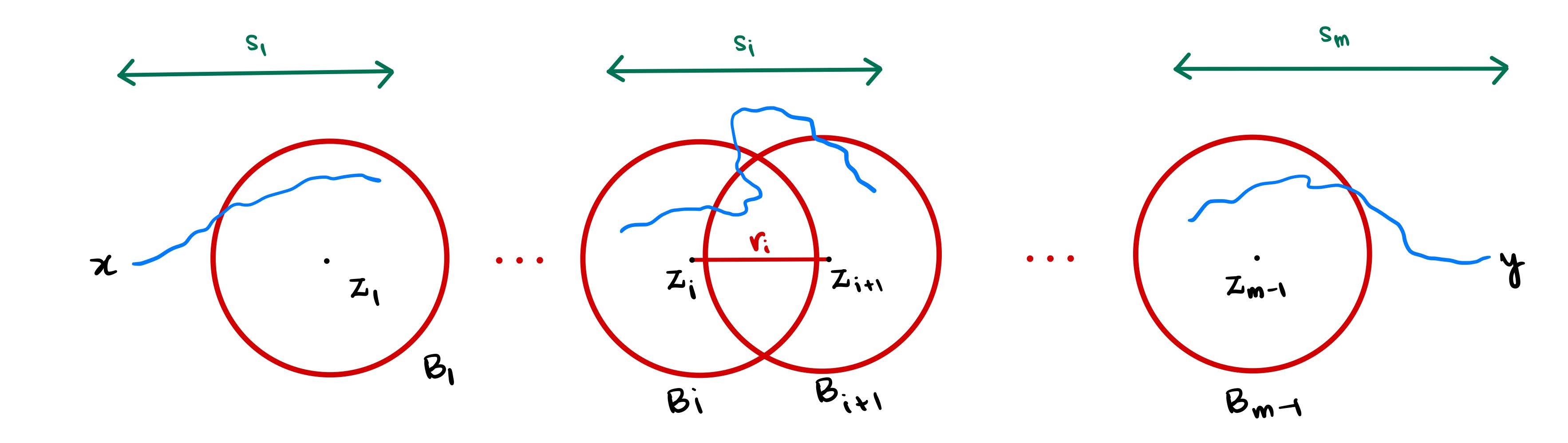}
  \caption{The basic idea of chaining of balls in the proof of   Case 3 in Proposition \ref{p:gb}(a) is to divide the random walk path into $m \approx \frac{R^2}{n}$ steps. This will imply that length of each segment $r_i$ above are approximately $\frac{R}{m}$,  the time taken $s_i$ above are approximately $\frac{n}{m}$,  and we can use \eqref{eq:llclt} for each segment $i$, $1 \leq i \leq m$ to obtain the result. The details of the proof ensures that there exists $m, r_i, s_i$ are all natural numbers. }
  \label{fig:cb}
\end{figure}

{\bf Case 3:} Let $\frac{2^6}{L^2} R \leq  n \leq \frac{R^2}{L^2}$. See Figure \ref{fig:cb} for the basic idea of the chaining balls of argument used in the proof here. Given $n$ in this region choose $m \in \N$ such that 
\begin{equation}
\frac{2^5R^2}{L^2n} \leq m \leq \frac{2^6R^2}{L^2n}. \label{eq:cm}\\
\end{equation} 

Note that because  $ n \geq \frac{2^6}{L^2} R$, we have $ 1 \leq m \leq R.$ Set $r = \lfloor \frac{R}{m} \rfloor$ and  $s = \lfloor \frac{n}{m} \rfloor.$ Now, there exists 
 \begin{eqnarray*}
\mbox{$z_1, z_2, \ldots, z_{m}$:}&x=z_0, \quad y = z_m,\quad d(z_i,z_{i+1}) \in \{r, r+1\},  \\
\mbox{and}&& \nonumber\\
\mbox{$s_1, s_2, \ldots, s_m$:}& s_i \in \{s,s+1\}, \quad \sum_{i=1}^m s_i =n.
\end{eqnarray*}
Let $B_i = B(z_i,r)$ for $ 0 \leq i \leq m-2$.
 By considering all possible paths from $x$ to $y$ and the definition of $p_n$ we have that
 
$$
p_n(x,y) + p_{n+1}(x,y) \geq \P^x(X_{s_1} \in B_1) \cdot  \prod_{i=1}^{m-2} \min_{w \in B_i}\P^w (X_{s_{i+1}} \in B_{i+1}) \cdot  \min_{w \in B_{m-1}} \P^w (X_{s_{m}} =y).  \label{eqa:step1}
$$
Next, by our choice of $m$, the  lower bound in \eqref{eq:cm} implies that
\begin{equation}
(3r+1)^2 \leq 16r^2 \leq 16 \frac{R^2}{m^2} \leq  L^2 \frac{n}{2m} \leq L^2 s. \label{eq:crsl}
\end{equation}
 Since we have that $d(w,v) \leq 3r+1$ for all $w \in B_i$ and $v \in B_{i+1}$, the lower bound for $s$ in \eqref{eq:crsl} implies that $s \geq \frac{1}{L^2} d(w,v)^2$. Hence from \eqref{eq:llclt} we have  that there exists $0 < c_1 < 1$ such that
\begin{equation} \label{eq:kest}
\min_{w \in B_i}\P^w (X_{s_{i+1}} \in B_{i+1})  \geq \frac{c_1}{s^{\frac{d}{2}}} \qquad \mbox{for all} \qquad 0 \leq i \leq m-2. 
\end{equation}	
Using the volume bound from \eqref{eq:va} and  \eqref{eq:kest} in \eqref{eqa:step1} we have
\begin{eqnarray}\label{eqa:kestb}
p_n(x,y)&\geq&  \left(\frac{c_1 V_1 r^d}{ \left( s_1 \right)^\frac{d}{2}}\right) \quad\cdot \quad \prod_{i=1}^{m-2}\left(\frac{c_1 V_1 r^d}{\left( s_i \right)^\frac{d}{2}}\right) \quad\cdot \quad  \left(\frac{c_1}{\left( s_m \right)^\frac{d}{2}}\right).
\end{eqnarray}
By our choice of $m$, the  upper bound in \eqref{eq:cm} implies that
\begin{equation}
(16)^2r^2 \geq (16)^2 \frac{R^2}{4m^2} \geq  64 L^2 \frac{n}{64 m} \geq  L^2 s. \label{eq:crsu}
\end{equation}
Using \eqref{eq:crsu}, the fact that $s_i \geq s$ for all $1\leq i \leq m-1$, and $s_m \leq n$ in \eqref{eqa:kestb} we have
\begin{equation}
p_n(x,y)\geq \frac{c_1}{16^d} \left( \frac{c_1L^d V_1}{16^d }\right)^{m-1}\frac{1}{n^\frac{d}{2}}
=\frac{c_2}{n^{\frac{d}{2}}} \exp(-c_3 m) \label{eq:kestf}
\end{equation}
for some $c_2, c_3 >0.$ Note that $c_3 >0$ because $V_1 \leq 2d$ from \eqref{eq:va}. Inserting the lower bound  of $m$  from \eqref{eq:cm} in \eqref{eq:kestf} implies \eqref{eq:lhkb}. 

 \end{proof}

  \begin{remark} The choice of the constant  $\frac{2^6}{L^2}$ in Case 3 is not adhoc. One could derive it post-facto in the proof as well. To use the Local Central Limit Theorem in the chain of balls of argument we need a choice of $m$  such that $$\mbox{$3r+1 \leq L \sqrt{s} \leq  16r,$ \qquad \mbox{\em ({\em see}  \eqref{eq:crsu} {\em and }  \eqref{eq:crsl})}.}$$ This can be achieved if $$ 4r \leq \frac{1}{L} \sqrt{\frac{n}{2m}} \mbox{ and } \frac{1}{L} \sqrt{\frac{n}{m}} \leq 8 \frac{R}{m},$$ as $ \frac{n}{2m} \leq s \leq \frac{n}{m}$ and $ \frac{R}{2m} \leq r \leq \frac{R}{m}.$ 
      This is possible if $m$ is an integer as in $$\frac{2^5R^2}{L^2n} \leq m \leq \frac{2^6R^2}{L^2n}, \qquad \mbox{ \em ({\em see} \eqref{eq:cm}).}$$ This forces the choice of the constant in Case 3.
      \end{remark}

\subsection{Upper Bound}

Recall that, $B= B(x_0,R)$ and $\tau_B = \inf\{ k \geq 0: X_k \in B^c\}$.
We note that $\tau_B$ is finite with probability one because
\begin{equation} \label{eq:tfwp1}
  \P(\tau_B>n)\leq (1-p)^{\lfloor \frac{R}{n} \rfloor}
\end{equation}
with (say) $p=(\frac{1}{2d})^{3R}$. We will need the following lemma for the proof of the upper bound.

\begin{lemma} \label{l:hbound} 
  $\P^{x}(\tau_B<n) \leq 2d \exp(-\frac{R^2}{4dn}).$
\end{lemma}
\begin{proof}
  We first reduce the question for the $d$ dimensional random walk  to a one dimensional walk as follows.
  Without loss of generality we may assume $x_0=0$ and consider $B= B(0,R)$. Let $\{X_n\}_{n \geq 0}$ be the simple random walk on $\Z^d$ with  $X_0=0$.  Now, let $\{Y_n\}_{n\geq 0}$ be the lazy random walk on $\Z$,  starting from $0$ with transition matrix given by $q$ 
  \[ q(u,v) = \left \{ \begin{array}{ll}
    \frac{1}{2d} &\mbox{ if }  u = v \pm 1,\\
    &\\
    \frac{d-1}{d} & \mbox{ if } u = v,\\
    &\\
    0& \mbox{otherwise.}
  \end{array} \right. \]
Let $X_n(k)$ denote the $k$th coordinate of $X_n$ for $1\leq k \leq d$.  Note that for each $k$,  $\{X_n(k)\}_{n \geq 0}$ have the same distribution as $\{Y_n\}_{n \geq 0}$. Also, if $d(X_n,0)>R$ then there is a coordinate $1 \leq k \leq d$ for which $X_n(k)\geq \frac{R}{d}$, consequently using a union bound
\begin{equation} \label{eq:reduction}
\P(\tau_{B}\leq n)\leq d  \P(\tau_{\frac{R}{d}}\leq n)
\end{equation}
where $\tau_S= \inf\{k \geq 0 : Y_k \not \in [-S,S]\}$ for any $S >0$.

We will understand $\tau_S$ for any $S=\frac{R}{d}$ and for that we define two auxiliary stopping times:$$ \tau_S^+ = \left \{ \begin{array}{ll}\tau_S & \mbox{ if } Y_{\tau_S}>0,\\
  \infty& \mbox{ otherwise,}\end{array} \right. \qquad \mbox{ and } \qquad \tau_S^- = \left \{ \begin{array}{ll}\tau_S & \mbox{ if } Y_{\tau_S}<0,\\ \infty & \mbox{ otherwise.} \end{array} \right. $$
Further it is easy to see that
$$\{\tau_S\leq n\} = \{\tau_S^+ \leq n\}\dot\cup \{\tau_S^+ \leq n\}$$ implying
$$\P(\tau_S\leq n)=\P(\tau_S^+ \leq n)+\P(\tau_S^- \leq n).$$
We will estimate $\P(\tau_S^+ \leq n)$ and the other term will follow similarly.

For $S \leq k \leq n,$ let $ A_k = \{\tau^+_S = k\}.$ Note that $\P(A_k) >0$ and $$\{\tau_S^+\leq n\} = \cup_{k=S}^n A_k.$$
Note that,
$$ \E[Y_n|A_k] = \E[Y_k + Y_n-Y_k|A_k] \geq S + \E[Y_n -Y_k]= S  $$
 Let $\lambda >0$ and define$$T_n=\exp(\lambda Y_n)\qquad \mbox{ for } S \leq n.$$
By Jensen's inequality,
\begin{equation*}
  \E[T_n] = \sum_{k=S}^n{\E[T_n|A_k]\P(A_k)} \geq \sum_{k=S}^n \exp(\lambda \E[Y_n|A_k]) \P(A_k)= \exp(\lambda S)  \P(\tau_S^+\leq n)
    \end{equation*}
This implies,
\begin{equation}\label{eq:Chernoff}
    \P(\tau_S^+\leq n)\leq \frac{\E[T_n]}{\exp(\lambda S)}.
\end{equation}
Now it is standard that $$\E[T_n] = \left(\frac{d-1}{d}+\frac{\exp(\lambda)}{2d}+\frac{\exp(-\lambda)}{2d}\right)^n.$$
For $0 < \lambda < 1,$ we obtain
$$\frac{d-1}{d}+\frac{\exp(\lambda)}{2d}+\frac{\exp(-\lambda)}{2d} = \frac{d-1}{d}+ \frac{1}{d}\sum_{k=0}^\infty \frac{\lambda^{2k}}{(2k)!} = 1 + \frac{\lambda^2}{d} \left(\sum_{k=1}^\infty \frac{\lambda^{2(k-1)}}{(2k)!} \right) \leq    1+\frac{\lambda^2}{d},$$
as $\sum_{k=1}^\infty \frac{\lambda^{2(k-1)}}{(2k)!}< \sum_{k=1}^\infty \frac{1}{(2k)!} < 1.$ Therefore,
\begin{align*}
    \E[T_n] \leq  \left(1+\frac{\lambda^2}{d}\right)^n  \leq \exp\left(n  \frac{\lambda^2}{d}\right). 
\end{align*}
Substituting this estimate in $\eqref{eq:Chernoff}$ and choosing $\lambda=\frac{dS}{2n}$ (note that $\lambda < 1 $) gives  
 
$$\P(\tau_S^+\leq n)\leq \exp\left(n\frac{\lambda^2}{d}-\lambda S\right) = \exp\left (-\frac{dS^2}{4n} \right).$$
Using symmetry of the distribution of the walk, we also have $$\P(\tau_S^-\leq n)\leq \exp\left(-\frac{dS^2}{4n}\right).$$ Therefore, we get
$$\P(\tau_S\leq n)\leq 2\exp\left(-\frac{dS^2}{4n}\right).$$ As $S=\frac{R}{d}$ in \eqref{eq:reduction} we get that,

$$\P(\tau_{B}\leq n)\leq d \times \P\left(\tau_{\frac{R}{d}}\leq n \right)\leq 2d\exp\left(-\frac{R^2}{4dn}\right),$$ as required. 

\end{proof}

We are now ready to prove Proposition \ref{p:gb} (b). The below proof follows the same structure as in the proof of \cite[Theorem 4.34]{MB17}.

\begin{proof}[Proof of Proposition \ref{p:gb} (b)]  Let $R = d(x,y)$. As in part (a),  we consider three cases.

{\bf Case 1:} $n <R$. Here $p_n(x,y) = 0$ and \eqref{eq:uhkb} follows immediately. 

{\bf Case 2:} $n \geq R^2$. This would imply $\frac{d^2(x,y)}{n} < 1$ which inturn implies that
$$\exp(-1)\leq \exp \left(-\frac{d^2(x,y)}{n}\right) \leq 1.$$ From above and \eqref{eq:ulclt} we have

  $$ p_n(x,y) \leq \frac{N_1}{n^{\frac{d}{2}}} \leq \frac{c_1}{n^{\frac{d}{2}}} \exp\left(-\frac{d^2(x,y)}{n}\right).$$
\eqref{eq:uhkb} follows.

{\bf Case 3:} $R \leq n \leq R^2$. Let $m \in \N$ be such that
\begin{equation}\label{eq:ucm}
\frac{n}{3} \leq m \leq \frac{2n}{3}.
\end{equation}
Define the set $A_x = \{ z \in \Z^d : d(x,z) \leq d(y,z)\}$ and $A_y = A_x^c$. Then
using time reversibility of the random walk we have
\begin{align}\label{eq:tr}
p_n(x,y) &= \P^x(X_m \in A_y, X_n =y) +\P^x(X_m \in A_x, X_n =y) \nonumber\\
&=\P^x(X_m \in A_y, X_n =y) +\P^y(X_{n-m} \in A_x, X_n =x)
\end{align}
Let ${\cal F}_m = \sigma(X_1, X_2, \ldots, X_m)$. Using the Markov Property we have that
\begin{align*}
\P^x(X_m \in A_y, X_n =y) & = \P^x(\tau_{B\left(x, \frac{R}{2} \right)} \leq m, X_m \in A_y, X_n =y)\\
& = \E^x \left( \E^x\left(1(\tau_{B\left(x, \frac{R}{2} \right)} \leq m, X_m \in A_y) \, 1(X_n =y) | {\cal F}_m \right)\right)\\
& = \E^x\left(1(\tau_{B\left(x, \frac{R}{2} \right)} \leq m, X_m \in A_y) \, \E^x\left( 1(X_n =y) | {\cal F}_m \right)\right)\\
 & = \E^x\left(1(\tau_{B\left(x, \frac{R}{2} \right)} \leq m, X_m \in A_y) \, \E^{X_m}(1(X_n =y)) \right)\\
&\leq\, \P^x(\tau_{B\left(x, \frac{R}{2} \right)} \leq m) \sup_{z \in A_y} p_{n-m}(z,y).
\end{align*} 
and similarly
$$ \P^y(X_{n-m} \in A_x, X_n =x) \, \leq \, \P^y(\tau_{B\left(y, \frac{R}{2}\right)} \leq n-m) \sup_{z \in A_x} p_{m}(x,y).$$
Using the above, along with Lemma \ref{l:hbound} and \eqref{eq:ulclt} in \eqref{eq:tr} we have that there exists $c_1,c_2,c_3, c_4 >0$ such that 
 \begin{align*}  
p_n(x,y) \leq  c_1 \exp\left(-c_2 \frac{R^2}{m}\right) \frac{N_1}{(n-m)^{\frac{d}{2}}} + c_3 \exp\left(-c_4 \frac{R^2}{n-m}\right) \frac{N_1}{m^{\frac{d}{2}}}.
 \end{align*}
\eqref{eq:uhkb} then follows using the bounds on $m$ from \eqref{eq:ucm}.
\end{proof}

\section{Uniform Green Inequality} \label{sec:ugi}

In this section we shall prove Proposition \ref{p:ugi}. To prove the
result we obtain Gaussian bounds for the $n$-step killed probability
function.

Recall that $B=B(x_0,R)$ and for any $x,y \in \Z^d$ the $n$-step killed probability function is given by,
    $$p^B_n(x,y) = \P^x(X_n = y, n < \tau_B).$$
Clearly, $p_n^B (x,y) \leq p_n(x,y).$ So the Gaussian upper
    bounds follows from Proposition \ref{p:gb}. The Gaussian lower bound is not that immediate and is presented below.

\begin{lemma}[Killed -Gaussian Lower Bound] \label{l:kdb} Let $ d \geq 1$.  $\exists A >0, C >0$ such that for all $R \geq 2$, $x_0 \in \Z^d$, and $B= B(x_0,R)$, $x,y \in B^\prime= B(x_0, \frac{R}{2})$, $\max\{ 1, d(x,y)\} \leq n \leq R^2$
        \begin{equation} \label{eq:kldb}
      p^B_n(x,y) +p^B_{n+1}(x,y) \geq \frac{A}{n^{\frac{d}{2}}}\exp\left(-C\frac{d^2(x,y)}{n}\right)     \end{equation}
\end{lemma}

\begin{proof}
  Let $x,y\in B(x_0, \frac{R}{2})$. We shall first prove that $ \exists\, 0 < \delta < \gamma <1$ such that  if
  $$\mbox{$d(x,y)\leq \delta R$ and $d(x,y)\leq n\leq \gamma R^2$ then \eqref{eq:kldb} holds.}$$ Then we shall use a chaining argument similar to proof of \eqref{eq:lhkb} for the general case. First we have
  \begin{align*}
    p_n(x,y) &=\P^x(X_n=y)\\
    &=\P^x(X_n=y, \tau_B>n)+\P^x(X_n=y, \tau_B\leq n)\\
    &= p_n^B(x,y) + \P^x(X_n=y, \tau_B\leq n).
    \end{align*}
As, $$\P^x(X_n=y, \tau_B \leq n)=\E^x(I_{(\tau_B\leq n)}\P^{X_{\tau_B}}(X_{n-\tau_B}=y))\leq \max_{z\in \partial B} \max_{1\leq m \leq n} \P^z(X_m=y)$$ we have 
\begin{equation}\label{eq:baB}
p_n^B(x,y)\geq p_n(x,y)-\max_{z\in \partial B} \max_{1\leq m \leq n} p_m(z,y).
\end{equation}
We also have a similar lower bound for $p^B_{n+1}(x,y)$. Now, for $ z \in \partial B$ and $y \in B(x_0, \frac{R}{2})$ we have $d(z,y)\geq \frac{R}{2}.$
So, Proposition \ref{p:gb} (b) implies that  $\exists c_1 >0$ and $c_2 >0$ such that $$\max_{z\in \partial B} \max_{1\leq m \leq n} p_m(z,y)\leq  \max_{1\leq m \leq n} \frac{c_1}{m^{\frac{d}{2}}} \exp\left(-c_2\frac{R^2}{m}\right).$$

Consider the function $f: \R \rightarrow \R$ given by $f(m)=\frac{c_1}{m^{\frac{d}{2}}}\exp\left(-c_2 \frac{R^2}{m}\right)$. It is easy to see that $f(\cdot)$ is unimodal with maximum at $m=cR^2$, for some constant $c >0$. Therefore if $n<cR^2$, then over the range $1\leq m \leq n$, $f(m)$ is maximised at $m=n$. So for $n < c R^2$, we have $$\max_{z\in \partial B} \max_{1\leq m \leq n} p_m(z,y)\leq  \frac{c_1}{n^{\frac{d}{2}}} \exp\left(-c_2\frac{R^2}{n}\right).$$ Using this and \eqref{eq:lhkb} in \eqref{eq:baB}, for $ \max\{1, d(x,y)\} \leq n \leq cR^2$ we have that there exists $c_3,c_4, c_5>0$ such that
 $$ p^B_n(x,y)+p^B_{n+1}(x,y)\geq c_3\frac{1}{n^{\frac{d}{2}}}\left(\exp\left(-c_4\frac{(d(x,y))^2}{n} \right)-\exp\left(-c_5\frac{R^2}{n}\right)\right).$$
So we can choose $0 < \delta < \gamma < 1$ (independent of $R$) such that  if $d(x,y)\leq \delta R$ and $ \max\{1, d(x,y)\} \leq n \leq \gamma R^2$, we get
\begin{equation}\label{eq:delta}
  p^B_n(x,y)+p^B_{n+1}(x,y)\geq c_6n^{-\frac{d}{2}}\exp\left(- c_7\frac{(d(x,y))^2}{n} \right).
\end{equation}
We are now ready to prove the general case.  Suppose $R_0 = \lfloor\frac{4}{\delta^2} \rfloor +1$.  Then there is $c_8 >0$ such that
for all $2 \leq R \leq R_0$, $x,y \in B(x_0, \frac{R}{2})$ and $
\max\{1, d(x,y)\} \leq n \leq R^2$, such that
$$p^B_n(x,y)+p^B_{n+1}(x,y)\geq \left(\frac{1}{2d}\right)^{n+1}=\frac{1}{2d}\exp(-c_8 n).$$
The above is true because for any $n \geq 1$ we can find a path of length $n$ (perhaps overlapping) inside $B(x_0, R)$ connecting $x$ and $y$. The constant $c_8 >0 $ will only depend on $R_0$ and $d$.
Then it is immediate to see that   $\exists A_1 >0, C_1 >0$ such that for all $2\leq R \leq R_0$, $x_0 \in \Z^d$, and $B= B(x_0,R)$, $x,y \in B^\prime= B(x_0, \frac{R}{2})$ , $\max\{ 1, d(x,y)\} \leq n \leq R^2$
        \begin{equation} \label{eq:kudbr0}
          p^B_n(x,y) +p^B_{n+1}(x,y) \geq \frac{A_1}{n^{\frac{d}{2}}}\exp\left(-C_1\frac{d^2(x,y)}{n}\right)     \end{equation}
        with $A_1 >0, C_1 >0$ depending only on $R_0$.

We shall assume now that $R >R_0$. We will apply a chaining of balls argument. Let $x,y \in B(x_0, \frac{R}{2})$, and $\delta R\leq u=d(x,y) \leq R$. We will consider two cases.

{\bf Case 1 :} $ u \leq n \leq \frac{4}{\delta}u$. There is $c_1>0$ such that 
$$p^B_n(x,y)+p^B_{n+1}(x,y)\geq \left(\frac{1}{2d}\right)^{n+1}\geq \exp(-c_1 n).$$
As before, it is easy to see in  the current range of $n$ that  $\exists A_2 >0, C_2 >0$ such that
        \begin{equation} \label{eq:kudb}
          p^B_n(x,y) +p^B_{n+1}(x,y) \geq \frac{A_2}{n^{\frac{d}{2}}}\exp\left(-C_2\frac{d^2(x,y)}{n}\right)
        \end{equation}
with $A_2 >0$ and $C_2>0$ depending only on $d, \delta >0$.

{\bf Case 2 :} $\frac{4}{\delta}u \leq n \leq R^2$.

Let $ m = \lfloor \frac{2}{\delta} \rfloor.$ Set $r = \lfloor \frac{u}{m} \rfloor$, and  $s = \lfloor \frac{n}{m} \rfloor.$ Now it is clear that as $ \frac{1}{\delta} \leq m \leq \frac{2}{\delta}$ and $R > R_0$ we have
$$ r \geq \frac{u}{2m} \geq \frac{\delta^2 R}{4} > 1 \mbox{ and } s \geq \frac{n}{2m} \geq \frac{\delta n}{2} \geq 2 u > 1.$$
There exists  \begin{eqnarray*}
z_1, z_2, \ldots, z_{m}:&x=z_0, \quad y = z_m,\quad d(z_i,z_{i+1}) \in \{r, r+1\},  \\
\mbox{and}&& \nonumber\\
s_1, s_2, \ldots, s_m:& s_i \in \{s,s+1\},\mbox{ and } \sum_{i=1}^m s_i =n. 
\end{eqnarray*}
Let $B_i = B(z_i,r)$ for $ 1 \leq i \leq m-1$.
 By considering all possible paths from $x$ to $y$ and definition of $p^B_n$ we have that
\begin{eqnarray}
p^B_n(x,y) + p^B_{n+1}(x,y) &\geq&  \P^{B,x}(X_{s_1} \in B_1) \cdot  \prod_{i=1}^{m-2} \min_{w \in B_i}\P^{B,w} (X_{s_{i+1}} \in B_{i+1})  \cdot\min_{w \in B_{m-1}} \P^{B,w} (X_{s_{m}} =y)  \nonumber \\\label{eqa:kbstep1}
\end{eqnarray}
Note that :
\begin{equation}
\label{eq:rsb} \frac{1}{\delta} \leq m \leq \frac{2}{\delta}, \quad \frac{u}{2m} \leq r \leq \frac{u}{m} \mbox{ and } \frac{n}{2m} \leq s \leq \frac{n}{m} 
 \end{equation}
and so
\begin{align*}
  & r \leq \delta R,\qquad
3r+1 < 4r < 4 \frac{u}{m} < \frac{n \delta}{m} < \frac{n}{2m} < s, \qquad 
\mbox{ and }s < \frac{n}{m} < \delta R^2.
 \end{align*}
Using the above, we can apply \eqref{eq:delta} to all the terms in \eqref{eqa:kbstep1} to obtain
\begin{eqnarray}
  p^B_n(x,y) + p^B_{n+1}(x,y) &\geq&  c_1^m \frac{\exp(-c_2\frac{mr^2}{s})}{s^\frac{d}{2}} \left(\frac{r}{\sqrt{s}}\right)^{d(m-1)}.\nonumber
\end{eqnarray}
Using \eqref{eq:rsb}  in the above we have
\begin{eqnarray}
  p^B_n(x,y) + p^B_{n+1}(x,y)  &\geq & \frac{\exp(-c_2\frac{u^2}{n})}{n^{\frac{d}{2}}} c_1^m m^{\frac{d}{2}} \left(\frac{u}{2 \sqrt{mn}} \right)^{d(m-1)}. \nonumber
\end{eqnarray}
Using the fact that $u \geq \delta R$ and $n \leq R^2$ we have .
\begin{eqnarray}
  p^B_n(x,y) + p^B_{n+1}(x,y)  &\geq&  c_1 \frac{\exp(-c_2\frac{u^2}{n})}{n^{\frac{d}{2}}} c_3^m m^{d-\frac{dm}{2}}. \nonumber
\end{eqnarray}
As $m = \lfloor \frac{2}{\delta} \rfloor$ we have
\begin{eqnarray}
  p^B_n(x,y) + p^B_{n+1}(x,y) &\geq &c_4 \frac{\exp(-c_2\frac{u^2}{n})}{n^{\frac{d}{2}}}
  \label{eqa:kbstep2}
\end{eqnarray}
So the proof is complete.

\end{proof}

\begin{proof}[Proof of Proposition \ref{p:ugi}]
Let $x, y \in \Z^d$. Recall that $r = \max\{1, d(x,y)\}.$ We begin with the lower bounds. From \eqref{eq:greenf} we have that
$$g_B(x,y)=\sum_{n=0}^{\infty}{p_n^B(x,y)}\geq \sum_{n=r^2}^{ R^2}{p_n^B(x,y)}.$$
Using Lemma \ref{l:kdb}, we get that there exists $c_1, c_2>0$  such that
\begin{equation} \label{eq:int} g_B(x,y)\geq c_1\sum_{n=r^2}^{ R^2}\frac{1}{n^{\frac{d}{2}}} \geq c_2 \int_{r^2}^{ R^2}{s^{-\frac{d}{2}}} ds.\end{equation}
In the last inequality we have used the fact that sum is a upper bound for the Reimann integral. In the case $d=2$, evaluating the integral in \eqref{eq:int} we obtain $$g_B(x,y) \geq 2\log(R/r).$$
In the case $d>2$, evaluating the integral in   \eqref{eq:int}  we obtain 
\begin{align*}
  g_B(x,y) &\geq \int_{r^2}^{R^2}{s^{-\frac{d}{2}} ds} =\frac{2}{d-2}\left(\frac{1}{r^{d-2}}-\frac{1}{ R^{d-2}}\right)\\ &\geq  \frac{2}{d-2}\left(\frac{1}{r^{d-2}}-\frac{1}{(2r)^{d-2}}\right)\\ &= \left(\frac{2}{d-2}\right)\left(1-\frac{1}{2^{d-2}}\right)\frac{1}{r^{d-2}},
\end{align*}
where we have used the assumption $r\leq  \frac{R}{2}$ in the second last inequality. Thus we have completed the proof of the lower bound.

Next, we prove the upper bound for the case $d>2$.  As $p_n^B(x,y)\leq p_n(x,y)$, we will use using the Gaussian upper bound \eqref{eq:uhkb} in Proposition \ref{p:gb}(b) to obtain

\begin{align*}
  g_B(x,x) &= \sum_{n=0}^{\infty}p^B_n(x,x) \leq \sum_{n=0}^{\infty}p_n(x,x)
  \leq 1 + c_1\sum_{n=1}^{\infty}\frac{1}{n^{\frac{d}{2}}} \leq c_2,
\end{align*}

as required in the upper bound  and for $x\neq y$,  we get

\begin{align*}
g(x,y)&=\sum_{n=0}^{\infty}{p_n(x,y)}\leq \sum_{n=0}^{\infty}{\frac{c_1}{n^{\frac{d}{2}}}} \exp{\left(-c_2\frac{r^2}{n}\right)}\leq  c_3\int_0^{\infty}{\frac{c_1}{s^{\frac{d}{2}}} \exp\left(-c_2\frac{r^2}{s}\right)ds}.
  \end{align*}

  Changing variables in the final integral using $t=\frac{s}{r^2}$ we obtain $$\int_0^{\infty}{\frac{c_1}{s^{\frac{d}{2}}} \exp\left(-c_2\frac{r^2}{s}\right)ds}= \frac{c_1}{r^{d-2}}\int_0^{\infty}{\frac{\exp\left(-\frac{c_2}{t}\right)}{s^{\frac{d}{2}}}}ds=\frac{c_4}{r^{d-2}},$$ for some $c_4>0$.

Finally, we prove the upper bound for $d=2$. Splitting $g_n(x,y)$ suitably and using $p_n^B(x,y)\leq p_n(x,y)$ we have that there exists  $c_1,c_2,c_3,c_4 >0$ such that 
\begin{align*}
      g_B(x,y)&= p_0^B(x,y)+\sum_{n=1}^{r^2}{p_n^B(x,y)}+\sum_{n=r^2+1}^{R^2}{p_n^B(x,y)}+\sum_{n=R^2}^{\infty}{p_n^B(x,y)}\\
      & \leq 1+\sum_{n=1}^{r^2}{p_n(x,y)}+\sum_{n=r^2}^{R^2}{p_n(x,y)}+\sum_{n=R^2}^{\infty} \P^x(\tau_B >n)\\
                  & \leq 1+\sum_{n=1}^{r^2}{p_n(x,y)}+\sum_{n=r^2}^{R^2}{p_n(x,y)}+\sum_{n=R^2}^{\infty}{c_2\exp\left(-c_2\frac{R}{n}\right)}\\
         &\leq 1+\sum_{n=1}^{r^2}{p_n(x,y)}+\sum_{n=r^2+1}^{R^2}{\frac{c_3}{n}}+c_4,
      \end{align*}
where we have used  \eqref{eq:tfwp1} in third inequality and \eqref{eq:ulclt} in the final one.
For the third term we have the estimate $$\sum_{n=r^2+1}^{R^2}{\frac{c_1}{n}}\leq c_3\log(\frac{R}{r})$$ and for the second term using \eqref{eq:uhkb} and using the bound with Riemann integral we have that there exists, $c_5, c_6 >0$ such that
$$\sum_{n=1}^{r^2}{p_n(x,y)}\leq c_5\int_{0}^{r^2}{t^{-1}\exp\left(-B\frac{r^2}{t}\right)}dt=c_6\int_0^1{s^{-1}\exp\left(-\frac{B}{s}\right)}ds=c_7.$$ So we get that there exists $c_7 >0$ such that $$g_B(x,y)\leq c_7\log\left(\frac{R}{r}\right),$$ whenever $R \geq 2r$. This completes the proof for the upper bound in the case $d=2$. 
\end{proof}

\section{Balayage Formula} \label{sec:blyge}

In this section we shall prove Proposition \ref{p:blyge}.  We begin with another classical tool, a probabilistic representation for solutions of the Dirichlet problem in a domain.

\begin{lemma}\label{l:dir} Let $R \geq 1, x_0 \in \Z^d$, $D \subset B(x_0,R),$ and $\phi: \partial B(x_0,R) \rightarrow \R_+$. Then $h: \overline{B(x_0,R)} \rightarrow \R_+$ given by
  \begin{equation}
    h(x) = \E^x [\phi(X_{\tau_D})]
  \end{equation}
  is a unique solution to 
  \begin{eqnarray}
    \Delta h = 0 && \mbox{on }  D, \nonumber\\
    h = \phi &&\mbox{on }  \partial D. \label{eq:dirichlet}
  \end{eqnarray}
  \end{lemma}

\begin{proof}[Proof of Lemma \ref{l:dir}]
By \eqref{eq:tfwp1}, we know that $\tau_{B(x_0,R)} < \infty$ with probability 1 and hence $\tau_D < \infty$  with probability 1. Then using the Markov property and the one-step  probability  we  have
\begin{align*}
h(x) &=  \E^x [\phi(X_{\tau_D})]\\
&= \sum_{y \in \Z^d} \E^x [\phi(X_{\tau_D}); X_1=y]\\
&= \sum_{y \in \Z^d} p_1(x,y)\E^{y} [\phi(X_{\tau_D})]\\
&= Ph(x)
\end{align*}
This along with the fact that $\P^y[\tau_D=0] = 1$ for all $y \in \partial D$ implies that $h$ solves \eqref{eq:dirichlet}. 

We will now show uniqueness. Suppose there are two solutions $h_1$ and $h_2$ to \eqref{eq:dirichlet}. Then $h = h_1 -h_2$ solves \eqref{eq:dirichlet} with $\phi \equiv 0$. Extend $h$ to be $0$ on all of $\overline{B(x_0,R)}^c.$ Then
\begin{align*}
\sum_{x \in \Z^d} \sum_{y \in \Z^d}p_1(x,y)(h(x) - h(y))^2  
&=\sum_{x \in \Z^d}\sum_{y \in \Z^d}p_1(x,y)(h(x) - h(y))h(x) - \sum_{x \in \Z^d}\sum_{y \in \Z^d}p_1(x,y)(h(x) - h(y))h(y)
\end{align*}
Since the terms in the summand are non-zero only for finitely many terms, we can interchange the sums and we have
\begin{align*}
\sum_{x \in \Z^d} \sum_{y \in \Z^d}p_1(x,y)(h(x) - h(y))^2  &= 2\sum_{x \in \Z^d}\sum_{y \in \Z^d}p_1(x,y)(h(x) - h(y))h(x)\\
&= -2\sum_{x \in \Z^d} h(x) \left(\sum_{y \in \Z^d}p_1(x,y)(h(y) -h(x))\right)\\
& = -2\sum_{x \in \Z^d} h(x) (Ph(x) -h(x))\\
& = -2\sum_{x \in \Z^d} h(x) \Delta h(x) \\
&=0.
\end{align*}
Thus $h \equiv 0$ and $h_1=h_2$.
 \end{proof}
 
\begin{remark}
  An alternative argument for uniqueness proof is the following. Let $h$ be as in the proof. The maximum principle for harmonic functions implies that the maximum and the minimum of $h$ is attained on the boundary, and hence $h$ must be zero. 
 \end{remark}

We are now ready to prove Proposition \ref{p:blyge}
\begin{proof}[Proof of Proposition \ref{p:blyge}]
  Define $h_A: \bar{B} \rightarrow \R_+$ by
  $$ h_A(x) = \E^x[ h(X_{T_A}); T_A < \tau_B].$$
  From the definition of $\tau_B$ and $T_A$ it is immediate that 
  \begin{equation}
    h_A(x) = \E^x[ 1_A(X_{\tau_{B \setminus A}})h(X_{\tau_{B \setminus A}})] = \E^x[\phi(X_{\tau_{B \setminus A}})],
  \end{equation}
  where  $\phi : \overline{B} \rightarrow \R_+$ is given by $\phi(x) = 1_A(x)h(x)$. By the definition of $h_A$ we have  $h_A \equiv h$ on $A$. This and Lemma \ref{l:dir} will imply that
  \begin{equation}\label{eq:hha}
    \Delta h_A = 0 \mbox{ on } A^\circ \cup B\setminus A
  \end{equation}
  For $x \in \partial_i A$, $h_A(x) = h(x) = Ph(x) \geq Ph_A(x)$. Thus we have 
  \begin{equation}\label{eq:hha1}
    \Delta h_A \leq 0 \mbox{ on } \partial_i A.    
  \end{equation}
Define $f: B \rightarrow \R_+$ by
$$ f = (I_B- P_1^B)h_A,$$
where $I_Bg(x) = 1_B(x)g(x)$ and $P_1^Bg(x) = \sum_{y \in B}p_1^B(x,y)g(x)$ for $g: \Z^d \rightarrow \R$. 
From definition of $h_A$ we have that $h_A \equiv 0$ on $\partial B$, so $$ \Delta h_A = (P-I)h_A = (P-I)(I_Bh_A) = P(I_Bh_A) -I_Bh_A = (P_1^B-I_B)h_A = - f.$$
Using \eqref{eq:hha} and \eqref{eq:hha1}, we have that $f \geq 0$ and $\mbox{support} (f) \subset \partial_i(A)$. Define for $k \geq 1$, $$P_k^Bg(x) = \sum_{y \in B}p_k^B(x,y)g(x)$$ for $g: \Z^d \rightarrow \R$. Now for $n \geq 1$,  by the definition of $f$ and $h_A$ we have,
\begin{eqnarray*}
  \left| \sum_{k=0}^n P^B_kf(x) -h_A(x) \right| &=& \left| \sum_{k=0}^n P^B_{k} (P^B-I)h_A(x)  -h_A(x) \right|   \\
  &=&\left|  \left[\sum_{k=0}^n \left(P^B_{k+1}h_A(x) -P^B_kh_A(x) \right) \right] -h_A(x) \right|    \\
  &=& P^B_{n+1}h_A(x)  \\
  &=& \sum_{y \in B}p_k^B(x,y)h_A(x)\\
  &\leq&  c_1 \P^x(\tau_B > n).
  \end{eqnarray*} 
Using \eqref{eq:tfwp1}, we have that for all $x \in B$,
  $$ \sum_{k=0}^\infty P_k^Bf(x)  = h_A(x).$$
  To conclude we have, 
  $$ \sum_{y \in \partial_i A} g_B(x,y) f(y) = \sum_{y \in B} g_B(x,y) f(y) = \sum_{n=0}^\infty P_n^Bf(x)  = h_A(x) = h(x).$$
\end{proof}

\end{document}